\documentclass[12pt]{article}
\usepackage{amssymb,amsmath,amsthm}

\usepackage{fullpage}
\usepackage{graphicx}
\usepackage{xcolor}

\usepackage{color}
\usepackage[dvipsnames]{xcolor}
\usepackage{xcolor}

\newtheorem{theorem}{Theorem}
\newtheorem{lemma}[theorem]{Lemma}
\newtheorem{claim}[theorem]{Claim}

\newtheorem{corollary}[theorem]{Corollary}
\newtheorem{proposition}[theorem]{Proposition}

\title{On the dib-chromatic number of a digraph}
\author{Juan José Montellano-Ballesteros\footnotemark[2] \and Christian Rubio-Montiel\footnotemark[3]}

\begin{document}
\maketitle

\def\thefootnote{\fnsymbol{footnote}}
\footnotetext[2]{Instituto de Matemáticas, Universidad Nacional Aut{\'o}noma de M{\' e}xico, CDMX, Mexico. {\tt juancho@im.unam.mx}.}
\footnotetext[3]{Divisi{\' o}n de Matem{\' a}ticas e Ingenier{\' i}a, FES Acatl{\' a}n, Universidad Nacional Aut{\'o}noma de M{\' e}xico, Naucalpan, Mexico. {\tt christian.rubio@acatlan.unam.mx}.}

\begin{abstract} 
An acyclic coloring of a digraph that maximizes the number of colors such that each color class has a vertex pointing to all other classes and a vertex pointing to it from all other classes is known as the dib-chromatic number of a digraph.

In this paper, we answer the question about the existence of the dib-chromatic number and study the dib-chromatic number of bipartite digraphs.
\end{abstract}


%%%%%%%%%%%%%%%%%%%%%%%%%%%%%%%%%%%%%%%%%%%%%%%%%%%%%%%%%%%%%

\section{Introduction}

Chromatic graph theory requires ways to optimize the search for its parameters; one way is through algorithms. For instance, we would like to perform an operation to reduce the number of colors in a graph with a given proper coloring. One way to achieve this is by recoloring its vertices, not necessarily in a greedy way, but by distributing the vertices of a single color to the rest of the color classes. Of course, such recoloring is not feasible if such a color class contains a vertex with neighbors to all other color classes; such a vertex is called a \textit{$b$-vertex}. A \emph{$b$-coloring} of a given graph is a coloring of the vertices such that in each color class there exists a $b$-vertex.
The \textit{$b$-chromatic number} of the graph $G$ is defined as the worst case for the number of colors used by this coloring heuristic method, equivalently defined as the maximum $k$ such that a proper $b$-coloring using $k$ colors is accepted, and denoted by $b(G)$. Note that, if $G$ is a graph with a proper coloring using $\chi(G)$ colors, such a coloring is a $b$-coloring.

The study of the $b$-chromatic number emerged in the late 1990s \cite{MR1670155} where its complexity was determined. Since then, several papers have expanded its study, for example \cite{MR2606622,MR1927071,MR2063820}. It has been of interest to extend the chromatic number to the theory of digraphs \cite{javiernol2024dibchromaticnumberdigraphs,MR3511873}, as it has been for other coloring parameters \cite{MR3875016,MR4426060,MR2998438,MR693366,MR3202296}. 

For instance, \emph{the dichromatic number} $\text{dc}$ is a generalization of the chromatic number for a graph, where a proper coloring in graphs corresponds to an \emph{acyclic coloring} in a digraph $D$ and the goal is to find the minimum number of colors $\text{dc}(D)$ used in this type of coloring introduced by Neumann-Lara \cite{MR693366} and by Erd{\H o}s \cite{MR593699} when they showed that in a tournament $T_n$ then $\text{dc}(T_n)=O(\frac{n}{\lg n})$, also see \cite{MR841316}.

In an acyclic coloring of $D$ with $\text{dc}(D)$ colors, there are arcs in both directions between each pair of color classes, giving rise to the extension of \emph{complete coloring} of graphs to digraphs \cite{MR2895432,MR3329642}. In \cite{MR3875016}, \emph{the achromatic number} $\text{dac}(D)$ of a digraph $D$ is defined as the maximum over the acyclic and complete colorings in $D$, showing results such as in a tournament $T_n$ then $\text{dac}(T_n)=\Omega(\frac{n}{\lg n})$.

To generalize the $b$-chromatic number from an algorithmic point of view, one could define a ``$b$-vertex'' with respect to an acyclic coloring of a digraph $D$ as a $u$ vertex colored $i$ such that $u\cup C_j$ contains a directed cycle for each color class $C_j$, $i\not=j$. However, if we consider a directed cycle $\overrightarrow{C}_3$, with an acyclic coloring using two colors, the singular color class is a ``$b$-vertex'' but the other color class has no ``$b$-vertices''.

In \cite{javiernol2024dibchromaticnumberdigraphs}, a vertex $u$ is called a \emph{$b^+$-vertex} with respect to a coloring $\Gamma$ of a digraph $D$ if $u$ is incident to a vertex colored with each color different from the color of $u$. Analogously, a vertex $v$ is a \emph{$b^-$-vertex} with respect to $\Gamma$ if $v$ is incident from a vertex colored with each color different from the color of $v$.

A \emph{$b$-coloring} of $D$ is a coloring such that each color class contains a $b^+$-vertex and a $b^-$-vertex.
The \emph{$\text{dib}$-chromatic number} of $D$, denoted by $\text{dib}(D)$, is the maximum $k$ such that $D$ admits an acyclic $b$-coloring with $k$ colors. The question then arises as to the existence of the $\text{dib}(D)$ for a given digraph $D$.

This paper is divided as follows. Section 2 presents the existence of the dib-chromatic number. Section 3 addresses the study of the dib-chromatic number in bipartite digraphs, including simple bipartite digraphs.

%%%%%%%%%%%%%%%%%%%%%%%%%%%%%%%%%%%%%%%%%%%%%%%%%%%%%%%%%%%%%

\section{Existence}

In this section, we will show that for any given digraph $D$, the dib-chromatic number $\text{dib}(D)$ of $D$ exists, by showing that $\text{dc}(D) \leq \text{dib}(D)$, where $\text{dc}(D)$ is the  dichromatic number of $D$. We consider finite, and loopless digraphs, i.e., digraphs where digons are permitted.

Let $D=(V(D), E(D))$ be a digraph and $r\geq 2$ be an integer. 
An acyclic coloring $\Gamma$ of $V(D)$ using $r$ colors will be called {\it $b$-irreducible} if there is no other acyclic coloring $\Gamma'$ of $V(D)$ using $r-1$ colors such that, for each $i\in [r-1]:= \{1,\dots,r-1\}$, $\Gamma^{-1}(i)\subseteq \Gamma'^{-1}(i)$.

%An acyclic $r$-coloring of $V(D)$ will be called {\it complete} if for any pair $\{j,i\} \subseteq [r]$, with $i\not= j$, we have that for every $x\in \Gamma^{-1}(i)$, $N_D(x)\cap \Gamma^{-1}(j) \not = \emptyset$. 

Given an acyclic coloring  of $V(D)$ using $r$ colors and $i\in [r]$, we will say that a vertex $x\in\Gamma^{-1}(i)$ is a \emph{$b$-vertex} if it is a $b^+$-vertex and a $b^-$-vertex, and a pair of vertices $\{x,y\} \subseteq \Gamma^{-1}(i)$ will be called a \emph{$b$-pair} if $x$ is a $b^+$-vertex and $y$ is a $b^-$-vertex.

Recall that the \emph{out-neighborhood} of a vertex $x$ of $D$ is $N^+_D(x)=\{y\in V(D)\colon xy \in E(D)\}$, and its \emph{in-neighborhood} is $N^-_D(x)=\{y\in V(D)\colon yx \in E(D)\}$.

\begin{lemma}\label{existe} Let $D$ be a digraph and $\Gamma: V(D)\rightarrow [r]$ be an acyclic coloring of $D$ using $r$ colors. If $\Gamma$ is $b$-irreducible, then $\Gamma$ is a $b$-coloring of $D$.
\end{lemma}
\begin{proof}  Let $D$ be a digraph, $\Gamma: V(D)\rightarrow [r]$ be an acyclic $b$-irreducible coloring of $D$ using $r$ colors and let us suppose there is no $b$-vertex nor $b$-pair in $\Gamma^{-1}(r)$.

For each $x\in \Gamma^{-1}(r)$ let  $V^-_b(x) = \{j\in[r-1] : N_D^-(x)\cap \Gamma^{-1}(j) = \emptyset\}$ and 
 $V^+_b(x) = \{j\in[r-1] : N_D^+(x)\cap  \Gamma^{-1}(j) = \emptyset\}$.

Observe that since there is no  $b$-vertices in $\Gamma^{-1}(r)$, it follows that for every $x\in \Gamma^{-1}(r)$ we have that $V^+_b(x) \not= \emptyset$ or $V^-_b(x) \not= \emptyset$.
 
\begin{claim} Either for every $x\in \Gamma^{-1}(r)$ we have $V^+_b(x) \not= \emptyset$, or for every $x\in \Gamma^{-1}(r)$ we have $V^-_b(x) \not= \emptyset$. \end{claim}
 
 \noindent Suppose there is $x\in \Gamma^{-1}(r)$ where $V^+_b(x) = \emptyset$. Hence, for every $j\in[r-1]$ it follows that  $N_D^+(x)\cap \Gamma^{-1}(j) \not= \emptyset$. That is, $x$ is a $b^+$-vertex. Since there is no $b$-pairs in $\Gamma^{-1}(r)$ it follows that for each $y\in \Gamma^{-1}(r)\setminus\{x\}$  there is  $i_y\in[r-1]$ such that   $N_D^-(y)\cap \Gamma^{-1}(i_y)= \emptyset$. That is, for every $y\in \Gamma^{-1}(r)\setminus\{x\}$ we have 
  $V^-_b(y) \not= \emptyset$. Finally, since $x$ is not a $b$-vertex, there is $i\in[r-1]$ such that   $N_D^-(x)\cap \Gamma^{-1}(i) = \emptyset$. Thus  $V^-_b(x) \not= \emptyset$ and the claim follows.\\

 By the previous claim, without loss of generality, we suppose that for every $x\in \Gamma^{-1}(r)$ we have $V^+_b(x) \not= \emptyset$.
 Let  $B_1 = \{x\in \Gamma^{-1}(r) : 1\in V^+_b(x)\}$ and, for $1< j\leq r-1$, let  
 $B_j = \{x\in \Gamma^{-1}(r) : j\in V^+_b(x)\}\setminus\bigcup\limits_{1\leq i<j} B_i$. 

 Observe that since for every $x\in \Gamma^{-1}(r)$ we have $V^+_b(x) \not= \emptyset$, it follows that $\{B_1, \dots, B_{r-1}\}$ is a set of pairwise disjoint sets whose union is  $\Gamma^{-1}(r)$.

 Let $\Gamma':V(D)\rightarrow [r-1]$ be a coloring of $D$ using $r-1$ colors defined as follows: Given $x\in V(D)$, $\Gamma'(x) = j$ whenever $\Gamma(x) = j$ or $x\in B_j$. 
 
\begin{claim} $\Gamma'$ is an acyclic coloring such that, for each $i\in [r-1]$, $\Gamma^{-1}[i]\subseteq \Gamma'^{-1}[i]$. \end{claim}
 
 \noindent By definition of $\Gamma'$, it follows that, for each $i\in [r-1]$, we have $\Gamma^{-1}[i]\subseteq \Gamma'^{-1}[i]$. Let $j\in[r-1]$. By definition, $\Gamma'^{-1}(j)= \Gamma^{-1}(j)\cup B_j$ with $B_j\subseteq \Gamma^{-1}(r)$. Since $\Gamma$ is acyclic, there is no cycle contained neither in $D[\Gamma^{-1}(j)]$ nor $D[B_j]$, and since for every $x\in B_j$, $N_D^+(x)\cap \Gamma^{-1}(j) = \emptyset$, it follows that there is no cycle in $D[\Gamma'^{-1}(j)]$. \\

From the last claim, it follows that $\Gamma$ is not a $b$-irreducible coloring which is a contradiction. Thus, for every $i\in [r]$ there is either a $b$-vertex or a $b$-pair and the result follows. \end{proof}

\begin{theorem}
Given a digraph $D$, $\text{dc}(D) \leq \text{dib}(D)$. 
\end{theorem}
 \begin{proof}
 Let $D$ be a digraph and let $\Gamma$ be an acyclic coloring of $D$ with $\text{dc}(D)$ colors. Clearly $\Gamma$ is a $b$-irreducible acyclic coloring which, by Lemma \ref{existe}, implies that $\Gamma$ is a $b$-coloring and therefore $\text{dc}(D) \leq \text{dib}(D)$. \end{proof}

\def\di #1{\overrightarrow{#1}}

%%%%%%%%%%%%%%%%%%%%%%%%%%%%%%%%%%%%%%%%%%%%%%%%%%%%%%%%%%%%%

\section{Bipartite digraphs}

An \textit{independent} set of vertices in a digraph $D$ is a set of vertices without any arcs between them. A \emph{bipartite} digraph is a digraph $D=(V(D),E(D))$ where $V$ is partitioned into the independent sets $A$ and $B$, for short denoted by $D=(A,B)$. In this section, we state some results on bipartite digraphs.
 
We recall that the \textit{independence number} $\beta(D)$ of $D$ is the maximum possible number of vertices of an independent set. 

A digraph $D$ is \emph{weak} (or \emph{weakly connected}) if for every pair $u,v \in V(D)$ there exists a $uv$-path (not necessarily a directed path), otherwise, $D$ is called \emph{disconnected}. The number of arcs in a shortest $uv$-path  is denoted by $d(u,v)$, and if $u$ and $v$ belong to different components of a disconnected digraph, then $d(u,v)= \infty$.
 
%A digraph $D$ is \emph{strong} (or strongly connected) if for every pair $u,v \in V(D)$ there exists a directed $uv$-path and a directed $vu$-path.

Let $\delta^+_D$ denote \emph{the minimum out-degree}, $\delta^-_D$ denote \emph{the minimum in-degree}, and $\delta_D=\min\{\delta^+_D,\delta^-_D\}$ denote the minimum degree of $D$, respectively.
% (or simply by $\delta^+$, $\delta^-$ and $\delta$, respectively, when the digraph $D$ is understood).

%\begin{proposition}
%Let $D=(A , B)$ be a non-trivial strong bipartite digraph, then $\text{dib}(D)\geq 2.$ 
%\end{proposition}
%\begin{proof}
%By hypothesis, there exists at least a vertex $u \in A$ and a vertex $v \in B$, then there is a directed $uv$-path and a directed $vu$-path. Therefore, $D$ has a directed cycle and the result follows.
%\end{proof}

\begin{proposition}\label{prop5}
Let $D=(A , B)$ be a bipartite digraph such that $\delta_D\geq 2$, then $\text{dib}(D)\geq 2.$ 
\end{proposition}
\begin{proof}
Consider a directed $uv$-path $P$ of maximum length, then vertex $v$ has an out-neighbor in $P$ thus $D$ contains a directed cycle and the result follows.
\end{proof}

The following result is given in \cite{javiernol2024dibchromaticnumberdigraphs}.

\begin{corollary}\label{cor5}\cite{javiernol2024dibchromaticnumberdigraphs}
For any digraph $D$ of order $n$, then $dib(D)\leq n-\beta(D)+1.$
\end{corollary}

\begin{lemma}\label{lemma7}
If $D=(A , B)$ is a connected bipartite digraph with $|A| = n\leq m= |B|$, then \[\text{dib}(D)\leq n+1 \]
\end{lemma}
\begin{proof}
By Corollary \ref{cor5}, we have $\text{dib}(D)\leq n+m-\beta(D)+1=n+m-max\{n,m\}+1=n+1.$
\end{proof}

\begin{lemma}\label{lema8}
Let $D=(A, B)$ be a bipartite digraph. If $D$ has a $b$-coloring with $k\geq 3$ colors, then a partition contains vertices of all $k$ colors and the other partition contains vertices of at least $k-1$ colors.
\end{lemma}
\begin{proof}
Since $k \geq 3$, there are at least two $b^+$-vertices at some partition, say $B$. Hence $A$ contains all the $k$ colors in its vertices. If $A$ contains a $b^+$-vertex then it is incident to $k-1$ different colors in $B$, otherwise, all $b^+$-vertices belong to $B$, then $B$ has vertices of all the $k$ colors.
\end{proof}

The following two results are generalizations of a result given in \cite{MR2063820}, which provide sufficient and necessary conditions for bipartite digraphs with $\text{dib}(D)>2$.

\begin{theorem}\label{teo9}
Let $D=(A, B)$ be a disconnected bipartite digraph such that $\delta_D\geq 2$ and it is the union of $D_1, \dots,D_r$ weak bipartite subdigraphs. The component $D_i=(A_i,B_i)$ where $A=\bigcup_{i=1}^rA_i$ and $B=\bigcup_{i=1}^rB_i$. Then $\text{dib}(D)>2$ if and only if 
\begin{enumerate}
\item $r\geq 3$; or

\item $r=2$ and at least one of $D_1$, $D_2$ is not a complete symmetric bipartite digraph.

\end{enumerate}
\end{theorem}
\begin{proof}
Suppose $D$ has no $b$-colorings with at least $3$ colors, by Proposition \ref{prop5}, $\text{dib}(D)=2$. Note that each component $D_i$ has at least 4 vertices.

\begin{enumerate}
\item Case $r\geq 3$. We can easily color a vertex in $D_i$ to be $b$-vertex of color $c_i$, for $i=1,2,3$.

\item Case $r= 2$. Assume $D_1$ is not complete bipartite. Let $u \in A$ and $v \in B$ be, such that $uv$ is not a digon in $D_1$. Color $u$ and $v$ by color $c_1$, the vertices of $A \setminus \{u\}$ by color $c_2$ and all vertices of $B \setminus \{v\}$ by color $c_3$. Because of the connectivity of $D_1$ there is a $b$-pair of color
$c_2$ in $A$ and a $b$-pair of color $c_3$ in $B$. Now, we can easily color a vertex in $D_2$ to be $b$-vertex of color $c_1$.

On the other hand, if $D_1$ and $D_2$ are complete symmetric bipartite digraphs, its dichromatic and diachromatic number is $2$, then $\text{dib}(D) = 2$.

\end{enumerate}
\end{proof}

Let $v$ be a vertex in $D$, let $\overline{N}\text{\,}^  +(v)=\{u\in V(D):vu \notin E(D), v\neq u \}$ and  $\overline{N}\text{\,}^-(v)=\{w\in V(D):wv \notin E(D), v\neq w \}$ denote the out-non-neighborhood and in-non-neighborhood of $v$, respectively. We define the sets $\overline{N}_{\cup}(v)= \overline{N}\text{\,}^+(v) \cup \overline{N}\text{\,}^-(v)$ and $\overline{N}_{\cap}(v)= \overline{N}\text{\,}^+(v) \cap \overline{N}\text{\,}^-(v)$.

\begin{theorem}\label{teo10}
Let $D=(A, B)$ be a weakly connected bipartite digraph such that $\delta_D\geq 2$.
\begin{enumerate}
\item If $\text{dib}(D)>2$ then $A \subseteq \bigcup _{v\in B} \overline{N}_{\cup}(v)$ or $B \subseteq \bigcup _{v\in A} \overline{N}_{\cup}(v)$; and
    
\item If $A \subseteq \bigcup _{v\in B} \overline{N}_{\cap}(v)$ or $B \subseteq \bigcup _{v\in A} \overline{N}_{\cap}(v)$ then $\text{dib}(D)>2$.
\end{enumerate}
\end{theorem}
\begin{proof}
\begin{enumerate}

\item We assume there exists a $b$-coloring using $k\geq 3$ colors of $D$. By Lemma \ref{lema8}, the set $A$ or $B$ has vertices of all $k$ colors. Without loss of generality, $A$ contains vertices $v_i$ colored $c_i$ for $i\in\{1,\dots,k\}$. If there is $w\in B\setminus \bigcup _{i=1}^{k} \overline{N}_{\cup}(v_i)$, then $A\cup \{w\}$ are the vertices of a symmetric star and the coloring is not acyclic, hence, we have $B \subseteq \bigcup _{i=1}^{k} \overline{N}_{\cup}(v_i)\subseteq \bigcup _{v\in A_1} \overline{N}_{\cup}(v)$.

\item Now, suppose $D$ has no $b$-colorings with at least $3$ colors, by Proposition \ref{prop5}, $\text{dib}(D)=2$. Without loss of generality, suppose $B \subseteq \cup_{v\in A} \overline{N}_{\cap}(v)$. 

Let $w_1$ a vertex of $B$. By assumption, there exists a vertex $v_1 \in A$ such that $w_1\in \overline{N}_{\cap}(v_1)$, then there are no arc between $v_1$ and $w_1$, so color $v_1$ and $\overline{N}_{\cap}(v_1)\cap B$ by $c_1$. Since $\delta(v_1)\geq 2$, there exists at least one uncolored vertex $w_2 \in B$ and a vertex $v_2 \in A \setminus \{v_1\}$ such that $w_2\in \overline{N}_{\cap}(v_2)$, then $v_2w_2 \notin E(D)$ and $w_2v_2 \notin E(D)$. Color $v_2$, $w_2$ and the uncolored vertices of $\overline{N}_{\cap}(v_2) \cap B$ by $c_2$. Since $\delta(w_1)\geq 2$, there exists at least one uncolored vertex $v \in A$. We continue with this procedure  %the same idea 
until all vertices of $B$ are colored. If there are uncolored vertices in $A$, color them all with an additional color $c_i$, with $i\geq 3$. 
We use $k\geq 3$ colors and at least $k-1$ color classes have non-empty intersection with the sets $A$ and $B$ by construction. 
Suppose that the coloring is not a $b$-coloring, then there is a color class without a $b$-vertex or a $b$-pair. Recolor the vertices of this color class obtaining an acyclic coloring with $k-1$ colors. Again at least $k-2$ color classes have non-empty intersection with the sets $A$ and $B$. Continue this process until a $b$-coloring
is obtained. Since a $b$-coloring using $2$ colors is not possible, a contradiction to the assumption. Hence the result is a $b$-coloring with at least $3$ colors.

\end{enumerate}
\end{proof}

\subsection{On simple bipartite digraphs}

Recall a simple digraph is a digraph without digons, that is, an oriented graph.

\begin{theorem}\label{d=3} Let $D$ be a simple bipartite  digraph.  If $\delta(D)\geq 2$ then $\text{dib}(D)\geq 3$.     
\end{theorem} 
\begin{proof} Let $D=(A,B)$ be a simple bipartite  digraph, with $\delta(D)\geq 2$ and suppose  $\text{dib}(D)\leq 2$. 
A (non-directed) path $P = (x_1, \dots, x_k)$ in $D$ will be called \emph{bad} if every $x_{2i}\in V(P)$, with $1\leq 2i\leq k$, $x_{2i}$ is a sink in $P$ (respectively, $x_{2i}$ is a source in $P$); and every $x_{2i-1}\in V(P)$, with $1\leq 2i-1\leq k$, $x_{2i-1}$ is a source in $P$ (respectively, $x_{2i-1}$ is a sink in $P$).

Let $P=(x_1, \dots, x_k)$ be a bad path in $D$ of maximum order. Without loss of generality, we will assume that every $x_{2i}\in V(P)$, with $1\leq 2i\leq k$, $x_{2i}$ is a sink in $P$; every $x_{2i-1}\in V(P)$, with $1\leq 2i-1\leq k$, $x_{2i-1}$ is a source in $P$ and that  the sources are in $A$ and the sinks are in $B$.

\begin{claim}\label{b8} $k\leq 7$.
\end{claim}
\begin{proof} Suppose there is a bad path $(x_1, x_2, \dots, x_k)$ in $D$ with $k\geq 8$.   Let $\Gamma: V(D) \rightarrow \{c_1, c_2, c_3\}$ defined as: $\Gamma(x_1) = \Gamma(x_4) = \Gamma(x_7) = c_1$; $\Gamma(x_2) = \Gamma(x_5) = \Gamma(x_8)= c_2$ and $\Gamma(x_3) = \Gamma(x_6) = c_3$. Also, for every $x\in A\setminus\{x_1, x_3, x_5, x_7\}$, $\Gamma(x) = c_1$ and for every $x\in B\setminus\{x_2, x_4, x_6, x_8\}$, $\Gamma(x) = c_3$. It is not hard to see that this is a $b$-coloring of $D$ with 3 colors, where $x_2, x_4$ and $x_6$ are $b^-$-vertices and $x_3, x_5$ and $x_7$ are $b^+$-vertices,  which is not possible. 
\end{proof}

\begin{claim}\label{b7} $k\leq 6$.
\end{claim}
\begin{proof} Suppose there is a bad path $P=(x_1, x_2, \dots, x_7)$ in $D$. By Claim \ref{b8} there is no bad path of order $8$, thus $N^+(x_7) \subseteq V(P)$ and since $\delta(D) \geq 2$,  either $x_7x_2\in E(D)$ or $x_7x_4\in E(D)$. If $x_7x_2\in E(D)$ then in $D$ there is the undirected cycle $C_6= (x_2, x_3, x_4, x_5, x_6, x_7)$.  Let $\Gamma: V(D) \rightarrow \{c_1, c_2, c_3\}$ defined as: $\Gamma(x_3) = \Gamma(x_6) = c_1$; $\Gamma(x_4) = \Gamma(x_7) = c_2$ and $\Gamma(x_2) = \Gamma(x_5) = c_3$. Also, for every $x\in A\setminus\{x_3, x_5, x_7\}$, $\Gamma(x) = c_1$ and for every $x\in B\setminus\{x_2, x_4, x_6\}$, $\Gamma(x) = c_3$. It is not hard to see that this is a $b$-coloring of $D$ with 3 colors, where $x_2, x_4$ and $x_6$ are $b^-$-vertices and $x_3, x_5$ and $x_7$ are $b^+$-vertices, which is not possible. 

Hence, $x_7x_4\in E(D)$ and, by an analogous argument than above, $x_1x_4\in E(D)$. Thus, $\{x_1, x_3, x_5, x_7\}\subseteq N^-(x_4)$ and since $\delta(D) \geq 2$, there is a pair $\{z_1, z_2\}\subseteq N^+(x_4)$ in $A\setminus\{x_1, x_3, x_5, x_7\}$. 
Let $\Gamma: V(D) \rightarrow \{c_1, c_2, c_3\}$ defined as: $\Gamma(x_1) = \Gamma(x_4)= \Gamma(x_5)= c_1$; $\Gamma(x_2) = \Gamma(x_7) =\Gamma(z_1) = c_2$ and $\Gamma(x_3) = \Gamma(x_6) = \Gamma(z_2) = c_3$. Also, for every $x\in A\setminus\{z_1, z_2, x_1, x_3, x_5, x_7\}$, $\Gamma(x) = c_3$ and for every $x\in B\setminus\{x_2, x_4, x_6\}$, $\Gamma(x) = c_2$. It is not hard to see that this is a $b$-coloring of $D$ with 3 colors, where $x_4$ is a $b$-vertex, $x_2$ and  $x_6$ are $b^-$-vertices; and $x_3$ and $x_7$ are $b^+$-vertices, which is not possible. 
\end{proof}

\begin{claim}\label{b6} $k\leq 5$.
\end{claim}
\begin{proof} Suppose there is a bad path $P=(x_1, x_2, \dots, x_6)$ in $D$. By Claim \ref{b7} there is no bad path of order $7$, thus $N^+(x_1) \subseteq V(P)$ and since $\delta(D) \geq 2$,  either $x_1x_6\in E(D)$ or $x_1x_4\in E(D)$. If $x_1x_6\in E(D)$ then in $D$ there is the undirected cycle $C_6= (x_1, x_2, x_3, x_4, x_5, x_6)$. As in Claim \ref{b7} we see that this is not possible. Thus  $x_1x_4\in E(D)$. In an analogous way we see that $x_3x_6\in E(D)$.  Let $\Gamma: V(D) \rightarrow \{c_1, c_2, c_3\}$ defined as: $\Gamma(x_1) = \Gamma(x_6) = c_1$; $\Gamma(x_3) = \Gamma(x_4) = c_2$ and $\Gamma(x_2) = \Gamma(x_5) = c_3$. Also, for every $x\in A\setminus\{x_1, x_3, x_5\}$, $\Gamma(x) = c_1$ and for every $x\in B\setminus\{x_2, x_4, x_6\}$, $\Gamma(x) = c_3$. It is not hard to see that this is a $b$-coloring of $D$ with 3 colors, where $x_1, x_3$ and $x_5$ are $b^+$-vertices; and $x_2, x_4$ and $x_6$ are $b^-$-vertices, which is not possible. 
\end{proof}

\begin{claim}\label{b5} $k\leq 4$.
\end{claim}
\begin{proof} Suppose there is a bad path $P=(x_1, x_2, \dots, x_5)$ in $D$. By Claim \ref{b6} there is no bad path of order $6$, thus $N^+(x_1) \subseteq V(P)$ and since $\delta(D) \geq 2$,  $x_1x_4\in E(D)$. In an analogous way, we see that $x_5x_2\in E(D)$. Moreover, $d^+(x_1) = d^+(x_3) = d^+(x_5)=2$. Thus, $\{x_1, x_3, x_5\}\subseteq N^-(x_2)$ and $\{x_1, x_3, x_5\}\subseteq N^-(x_4)$ and therefore $N^+(x_2), N^+(x_4)\subseteq A\setminus \{x_1, x_3, x_5\}$. Let $\{z_1, z_2\}\subseteq N^+(x_2)$ and $\{w_1, w_2\}\subseteq N^+(x_4)$. 

\noindent {\bf Case 1.} $z_1= w_1$ and $z_2 = w_2$.

\noindent Since $\delta(D) \geq 2$ and $\{x_2, x_4\}\subseteq N^-(z_1)$, there is a pair $\{y_1, y_2\}\subseteq N^+(z_1)$ such that $\{y_1, y_2\} \subseteq B\setminus\{x_2, x_4\}$. Let $\Gamma: V(D) \rightarrow \{c_1, c_2, c_3\}$ defined as: $\Gamma(x_2) =\Gamma(x_5) = \Gamma(z_1) = c_1$; $\Gamma(x_1)  = \Gamma(x_4)  = \Gamma(y_2) =  c_2$ and $\Gamma(x_3) = \Gamma(y_1) = \Gamma(z_2) = c_3$. Also, for every $x\in A\setminus\{z_1, z_2, x_1, x_3, x_5\}$, $\Gamma(x) = c_1$ and for every $x\in B\setminus\{x_2, x_4, y_1, y_2\}$, $\Gamma(x) = c_2$. It is not hard to see that this is a $b$-coloring of $D$ with 3 colors, where $x_4$ is a  $b$-vertex,  $x_3$ and $z_1$ are $b^+$-vertices; and $x_2$ and $z_2$ are $b^-$-vertices,  which is not possible. 

\noindent {\bf Case 2.} $z_1= w_1$ and $z_2 \not=  w_2$.

\noindent  Let $\Gamma: V(D) \rightarrow \{c_1, c_2, c_3\}$ defined as: $\Gamma(x_1) = \Gamma(x_4) = \Gamma(z_2) = c_1$; $\Gamma(x_5)  = \Gamma(z_1) = c_2$ and $\Gamma(x_2) = \Gamma(x_3) = \Gamma(w_2) = c_3$. Also, for every $x\in A\setminus\{z_1, z_2, w_2, x_1, x_3, x_5\}$, $\Gamma(x) = c_3$ and for every $x\in B\setminus\{x_2, x_4\}$, $\Gamma(x) = c_1$. It is not hard to see that this is a $b$-coloring of $D$ with 3 colors, where $x_2$ and $x_4$ are  $b$-vertex, and $x_5$ and $z_1$ is a $b^+$-vertex and a $b^-$-vertex, respectively; which is not possible. 

\noindent {\bf Case 3.}  $\{z_1, z_2\}\cap \{w_1, w_2\} = \emptyset$.

\noindent  Since $w_1 \not \in N^+(x_2)$ and  $\delta(D) \geq 2$ there is $y_1\in N^-(w_1)\cap (B\setminus \{x_2, x_4\})$. 
Let $\Gamma: V(D) \rightarrow \{c_1, c_2, c_3\}$ defined as: $\Gamma(x_1) =\Gamma(x_4) = \Gamma(z_1) = c_1$; $\Gamma(x_2)  = \Gamma(x_3)  = \Gamma(y_1) =\Gamma(w_2) =  c_2$ and $\Gamma(x_5) = \Gamma(w_1) = \Gamma(z_2) = c_3$. Also, for every $x\in A\setminus\{z_1, z_2, w_1, w_2, x_1, x_3, x_5\}$, $\Gamma(x) = c_1$ and for every $x\in B\setminus\{x_2, x_4, y_1\}$, $\Gamma(x) = c_2$. It is not hard to see that the chromatic classes of colors $c_1$ and $c_3$ are acyclic. For the chromatic class of color $c_2$, observe that the only vertices of color $c_2$ in $A$ are $x_3$ and $w_2$. Thus, any directed cycle of color $c_2$ is a square and contains the vertices $\{x_3, w_2\}$. Since $d^+(x_3)= 2$ and $\Gamma (x_4)= c_1$, any directed cycle of color $c_2$ must contain the path $(x_3, x_2, w_2)$ which is not possible since $x_2w_2\not\in E(D)$. Hence, $\Gamma$ is a $b$-coloring of $D$ with 3 colors, where $x_2$ and $x_4$ are  $b$-vertices, and $x_5$ and $w_1$ is a $b^+$-vertex and a $b^-$-vertex, respectively; which is not possible.
\end{proof}

\begin{claim}\label{2reg} $D$ is $2$-regular.
\end{claim}
\begin{proof} Let $x\in V(D)$ with $d^+(x) =r\geq 2$ and let $N^+(x) = \{y_1, \dots, y_r\}$. By Claim \ref{b5} it follows that for some $z\in V(D)\setminus \{x, y_1, \dots, y_r\}$,  $\bigcap\limits_{y\in N^+(x)} N^-(y) = \{x, z\}$, which again, by Claim \ref{b5}, implies that $r=2$ (in other case, there is a bad path of order $5$ where $x_1, x_3$ and $x_5$ are sinks, and $x_2$, $x_4$ are sources). Thus, for every $x\in V(D)$, $d^+(x) =2$ and since $\delta(D) \geq 2$, the claim follows.
\end{proof}

\begin{claim}\label{b4} There are no bad paths of order $4$.
\end{claim}
\begin{proof}
Let $P=(x_1, x_2, x_3, x_4)$ be a bad path in $D$. By Claim \ref{b5} it follows that $x_1x_4\in E(D)$. By Claim \ref{2reg}, $\delta(D) = 2$ and therefore we see that $N^+(x_2) = \{z_1, z_2\}$ and $N^+(x_4) = \{w_1, w_2\}$ where $N^+(x_2), N^+(x_4) \subseteq A\setminus\{x_1, x_3\}$. Again, by Claim \ref{b5}, either $w_1= z_1$ and $z_2=w_2$;  or $\{z_1, z_2\}\cap \{w_1, w_2\}=\emptyset$. 

\noindent {\bf Case 1.} $\{z_1, z_2\}\cap \{w_1, w_2\}=\emptyset$.

\noindent By Claim \ref{b5} there is a pair  $\{y_1, y_2\}\subseteq B\setminus \{x_2, x_4\}$ such that $N^+(y_1) = \{z_1, z_2\}$ and $N^+(y_2) = \{w_1, w_2\}$. Let $\Gamma: V(D) \rightarrow \{c_1, c_2, c_3\}$ defined as: $\Gamma(x_1) =\Gamma(y_2) = \Gamma(z_1) = \Gamma(w_1) = c_1$; $\Gamma(x_2)  = \Gamma(x_3)  = \Gamma(w_2) = c_2$ and $\Gamma(x_4) = \Gamma(y_1) = \Gamma(z_2) = c_3$.  Also, for every $x\in A\setminus\{z_1, z_2, w_1, w_2, x_1, x_3, x_5\}$, $\Gamma(x) = c_2$ and for every $x\in B\setminus\{x_2, x_4, y_1, y_2\}$, $\Gamma(x) = c_3$. It is not hard to see that $\Gamma$ is a $b$-coloring of $D$ with 3 colors, where  $x_4$ is a  $b$-vertex;  $x_1$ and $x_2$ are $b^+$-vertices; and $z_1$ and $w_2$  are $b^-$-vertices, which is not possible.

\noindent {\bf Case 2.} $w_1= z_1$ and $z_2=w_2$.

\noindent Since $\delta(D) = 2$, $N^+(z_1)= \{y_1, y_2\}$ and $N^-(x_1)= \{y_3, y_4\}$ where $N^+(z_1), N^-(x_1) \subseteq B\setminus \{x_2, x_4\}$. 

Let $\Gamma: V(D) \rightarrow \{c_1, c_2, c_3\}$ defined as: $\Gamma(x_4)  = c_1$; $\Gamma(x_1)  = \Gamma(x_2)  = \Gamma(z_1) = c_2$ and $\Gamma(x_3) = \Gamma(z_2) = c_3$.  For the vertices $\{y_1, y_2, y_3, y_4\}$, it is possible that either $\{y_1,y_2\}= \{y_3, y_4\}$; or $\{y_1,y_2\}\cap \{y_3, y_4\}\emptyset$ or $y_1=y_3$ and $y_2\not=y_4$. In any case, color the vertices with colors $c_1$ and $c_3$ in a way that both colors are present in $N^-(x_1)$ an $N^+(z_1)$. 

Also, for every $x\in A\setminus\{z_1, z_2, x_1, x_3, x_5\}$, $\Gamma(x) = c_2$ and for every $x\in B\setminus\{x_2, x_4, y_1, y_2, y_3, y_4\}$, $\Gamma(x) = c_1$. It is not hard to see that $\Gamma$ is a $b$-coloring of $D$ with 3 colors, where  $x_4$ is a  $b$-vertex;  $x_3$ and $z_1$ are $b^+$-vertices; and $x_1$ and $z_2$  are $b^-$-vertices, which is not possible. \end{proof}

From Claim \ref{b4} we see that if $dib(D)\leq 2$, then there is no bad path of order $4$. Since $\delta(D)\geq 2$, this is not possible, and the result follows. 

\end{proof}

\begin{corollary} Let $D$ be a simple bipartite $2$-regular digraph. Then $dib(D) = 3$.
\end{corollary}
\begin{proof} By Theorem \ref{d=3} we see that $dib(D)\geq 3$, and it is not hard to see that $dib(D)\leq \max\limits_{x\in V(D)}\{d^+_D(x), d^-_D(x)\}+1 = 3$ and the result follows. 
    \end{proof}

From lemmas \ref{lemma7} and \ref{lema8} we obtain the following upper bound.

\begin{corollary}
Let $D =(A, B)$ be a simple bipartite digraph with $|A| = n\leq m = |B|$. If there is no source or there is no sink in $B$,  $\text{dib}(D)\leq n$.
\end{corollary}
\begin{proof} By Lemma \ref{lemma7},  $\text{dib}(D)\leq n+1$. If $\text{dib}(D)= n+1$,
by Lemma \ref{lema8}, there are no vertices of some color $c_{n+1}$ in $A$, and $b^+$-vertices and $b^-$-vertices of $c_1\dots, c_{n}$ belong to $A$. Hence, each vertex in $A$ is a $b$-vertex, and the  $b^+$-vertex and $b^-$-vertex of color $c_{n+1}$ are a source and a sink in $B$, respectively.
\end{proof}

Now we present some lower bounds. For this, given a digraph $D =(A, B)$, let $\delta_A = \min\{d^+_D(x), d^-_D(x) : x\in A\}$ and  $\delta_B = \min\{d^+_D(x), d^-_D(x) : x\in B\}$.

\begin{theorem}\label{partitions} Let $D =(A, B)$ be a simple bipartite  digraph, with $|A| = n\leq m = |B|$. If 
$$2n^2 <  \Big(\frac{m}{m-\delta_A}\Big)^{\lfloor\frac{m}{n}\rfloor}$$ then $dib(D)\geq n$. Moreover, there is a $b$-coloring of $D$ with $n$ colors such that every vertex in $A$ is a $b$-vertex. 
\end{theorem}
\begin{proof}
Let $D =(A, B)$ be a simple bipartite  digraph, with $|A| = n\leq m = |B|$, and assume that $B=[m]$. Let $P^*_n(m)$ be the set of $n$-partitions of $[m]$ such that each part has cardinality either $\lfloor\frac{m}{n}\rfloor$ or $\lceil\frac{m}{n}\rceil$ and let  $re(m,n)$ be the residue of $m$ modulo $n$. Observe that for every $n$-partition in $P^*_n(m)$, the number of parts of order $\lfloor\frac{m}{n}\rfloor$ is $n-re(m,n)$.  First we will assume that $re(n,m) \not=0$. 

It follows that ${{m}\choose {\lceil\frac{m}{n}\rceil}}|P^*_{n-1}(m-\lceil\frac{m}{n}\rceil)| = re(m,n) |P^*_n(m)|$ and that ${{m}\choose {\lfloor\frac{m}{n}\rfloor}}|P^*_{n-1}(m-\lfloor\frac{m}{n}\rfloor)| = (n-re(m,n)) |P^*_n(m)|$.  Thus 
\begin{equation}\label{nP} n|P^*_n(m)| = {{m}\choose {\lfloor\frac{m}{n}\rfloor}}|P^*_{n-1}(m-\lfloor\frac{m}{n}\rfloor)|+ {{m}\choose {\lceil\frac{m}{n}\rceil}}|P^*_{n-1}(m-\lceil\frac{m}{n}\rceil)|.\end{equation}

On the one hand, observe that $\frac{{{m}\choose{\lfloor\frac{m}{n}\rfloor}}}{{{m-\delta_A}\choose{\lfloor\frac{m}{n}\rfloor}}}=  \frac{m!(m-\delta_A-\lfloor\frac{m}{n}\rfloor)!}{(m-\lfloor\frac{m}{n}\rfloor)!(m-\delta_A)!}\geq (\frac{m}{m-\delta_A})^{\lfloor\frac{m}{n}\rfloor}$ and, similarly, $\frac{{{m}\choose{\lceil\frac{m}{n}\rceil}}}{{{m-\delta_A}\choose{\lceil\frac{m}{n}\rceil}}} \geq (\frac{m}{m-\delta_A})^{\lceil\frac{m}{n}\rceil}$. If $2n^2 <  (\frac{m}{m-\delta_A})^{\lfloor\frac{m}{n}\rfloor} $ then also $2n^2 <(\frac{m}{m-\delta_A})^{\lceil\frac{m}{n}\rceil}$ and therefore $2n^2< min\{\frac{{{m}\choose{\lfloor\frac{m}{n}\rfloor}}}{{{m-\delta_A}\choose{\lfloor\frac{m}{n}\rfloor}}}, \frac{{{m}\choose{\lceil\frac{m}{n}\rceil}}}{{{m-\delta_A}\choose{\lceil\frac{m}{n}\rceil}}}\}$.  Thus, by (\ref{nP}) we see that \begin{equation}\label{1} n|P^*_n(m)| >  2n^2{{m-\delta_A}\choose{\lfloor\frac{m}{n}\rfloor}}|P^*_{n-1}(m-\lfloor\frac{m}{n}\rfloor)|+ 2n^2{{m-\delta_A}\choose{\lceil\frac{m}{n}\rceil}}|P^*_{n-1}(m-\lceil\frac{m}{n}\rceil)|.\end{equation}

On the other hand, for every $x\in A$,  the number of $n$-partitions  $\{Y_1, \dots, Y_n\}$  of $P^*_n(m)$ such that there is a part $Y_j$ of order $\lfloor\frac{m}{n}\rfloor$ such that either $N_D^+(x)\cap Y_j=\emptyset$ or $N_D^-(x)\cap Y_j=\emptyset$ is at most ${{|B\setminus N^+_D(x)|}\choose{\lfloor\frac{m}{n}\rfloor}}|P^*_{n-1}(m-\lfloor\frac{m}{n}\rfloor)| + {{|B\setminus N^-_D(x)|}\choose{\lfloor\frac{m}{n}\rfloor}}|P^*_{n-1}(m-\lfloor\frac{m}{n}\rfloor)|\leq 2{{m-\delta_A}\choose{\lfloor\frac{m}{n}\rfloor}}|P^*_{n-1}(m-\lfloor\frac{m}{n}\rfloor)|.$ In a similar way we see that the number of $n$-partitions  $\{Y_1, \dots, Y_n\}$  of $P^*_n(m)$ such that there is a part $Y_j$ of order $\lceil\frac{m}{n}\rceil$ such that either $N_D^+(x)\cap Y_j=\emptyset$ or $N_D^-(x)\cap Y_j=\emptyset$ is at most $2{{m-\delta_A}\choose{\lceil\frac{m}{n}\rceil}}|P^*_{n-1}(m-\lceil\frac{m}{n}\rceil)|$. Hence, for each vertex $x\in A$, the number of elements $\{Y_1, \dots, Y_n\}$  of $P^*_n(m)$ such that there is a part $Y_j$  such that either $N_D^+(x)\cap Y_j=\emptyset$ or $N_D^-(x)\cap Y_j=\emptyset$ is at most $2\Big({{m-\delta_A}\choose{\lfloor\frac{m}{n}\rfloor}}|P^*_{n-1}(m-\lfloor\frac{m}{n}\rfloor)|+ {{m-\delta_A}\choose{\lceil\frac{m}{n}\rceil}}|P^*_{n-1}(m-\lceil\frac{m}{n}\rceil)\Big)$. 

By (\ref{1}) we see that $|P^*_n(m)| > 2n\Big({{m-\delta_A}\choose{\lceil\frac{m}{n}\rceil}}|P^*_{n-1}(m-\lceil\frac{m}{n}\rceil)| + {{m-\delta_A}\choose{\lfloor\frac{m}{n}\rfloor}}|P^*_{n-1}(m-\lfloor\frac{m}{n}\rfloor)\Big)$ which implies there is a partition $\{Y_1, \dots, Y_n\}$  in $P^*_n(m)$ such that for every $x\in A$ we have that for every $j\in [n]$, $N_D^+(x)\cap Y_j\not=\emptyset$ and $N_D^-(x)\cap Y_j\not=\emptyset$. 

Its is not hard to see that the coloring $\Gamma: V(D)\rightarrow [n]$ where each vertex in $A$ receives different color; and for each vertex $y\in B$, $y$ receives color $i\in [n]$ if and only if $y\in Y_i$, is a $b$-coloring of $D$ where every vertex in $A$ is a $b$-vertex. 

The proof for the case when $re(m,n)=0$ is similar, and the result follows. \end{proof}

\begin{corollary}
     Let $D$ be a simple bipartite  digraph, with $|A| = n\leq m = |B|$. Let $p$ be such that $\lfloor\frac{m}{n}\rfloor=p(1+2log_2(n))$. If $\delta_A > m\Big(1-\frac{1}{2^\frac{1}{p}}\Big)$  then $dib(D)\geq n$.    
\end{corollary}
\begin{proof}  If $\delta_A > m\Big(1-\frac{1}{2^\frac{1}{p}}\Big)$ then $\Big( \frac{m}{m-\delta_A}\Big)^{\lfloor\frac{m}{n}\rfloor} > \Big( \frac{m}{m-m\Big(1-\frac{1}{2^\frac{1}{p}}\Big)}\Big)^{\lfloor\frac{m}{n}\rfloor} = \Big( 2^\frac{1}{p}\Big)^{\lfloor\frac{m}{n}\rfloor} = 2^\frac{\lfloor\frac{m}{n}\rfloor}{p}= 2^{1+2log_2(n)} = 2n^2$ and from Theorem \ref{partitions} the result follows.   
\end{proof}

\begin{theorem} Let $D =(A, B)$ be a simple 
digraph, with $|A| = n\leq m= |B|$.  If $\delta_A\geq 2n(n-1)$, then $\text{dib}(D)\geq n$. Moreover, if $D$ is an orientation of  $K_{n,m}$, with $\delta_A\geq n^2$, then $\text{dib}(D)\geq n$.
\end{theorem}
\begin{proof} Let  $A=\{x_1, \dots, x_n\}$. First, assume that  $D$ is an orientation of $K_{n,m}$. Colored $x_1$ with color $c_1$ and $x_2$ with color $c_2$. Choose a vertex in $N^+(x_1)$ and another in $N^-(x_1)$  and assign them color $c_2$. Similarly, choose a vertex in $N^+(x_2)$ and another in $N^-(x_2)$  and assign them color $c_1$. 
Colored $x_3$ with color $c_3$. Since $\{N^+(x_3),N^-(x_3)\}$ is a partition of $B$, in $N^+(x_3) \cup N^-(x_3)$ there are already vertices with color $c_1$ and color $c_2$. Thus, we only need to choose at most one new vertex to be of color $c_1$ and another new vertex to be of color $c_2$ such that in $N^+(x_3)$ there are vertices of color $c_1$ and $c_2$ as well as in $N^-(x_3)$. Now, choose a new vertex in $N^+(x_1)$ and another in $N^-(x_1)$  and assign them color $c_3$. In this way, in $N^+(x_1)$ there are vertices of color $c_2$ and $c_3$ as well as in $N^-(x_1)$. Since $\{N^+(x_2),N^-(x_2)\}$ is a partition of $B$, either in $N^+(x_2)$ or in $N^-(x_2)$  there is already a vertex of color $c_3$. So, we only have to choose at most one new vertex to be of color $c_3$ so that in $N^+(x_2)$ there are vertices of color $c_1$ and $c_3$ as well as in $N^-(x_2)$. So far we have colored at most $9$ vertices in $B$. 
Following this procedure, assume that for $3\leq r<n$, for every $i\in [r]$ $x_i$ receives color $c_i$; for each $i\in[r]$ there are vertices of colors $\{c_1, \dots, c_r\}\setminus\{c_i\}$ in $N^+(x_i)$ and in $N^-(x_i)$; and that there are at most $ r^2$ colored vertices in $B$. Recall that $n^2 \leq \delta_A$, hence for each $x_i\in A$, with $i\in [r]$, in $N^+(x_i)$ (as well as in $N^-(x_i)$) there are at least $n^2-r^2$ vertices with no color. 

Let colored $x_{r+1}$ with color $c_{r+1}$. In  $N^+(x_{r+1})\cup N^-(x_{r+1})$ there are already vertices of color $c_1,\dots, c_{r}$. Thus, for each color $c_i$, with $i\in[r]$, we only have to choose at most  one new vertex in $B$ and  assign color $c_i$ to it, so that there are vertices of colors $\{c_1, \dots, c_r\}$ in $N^+(x_{r+1})$ and in $N^-(x_{r+1})$. Now choose any new vertex $y$ in $B$ and color it with color $c_{r+1}$. For every $i\in [r]$, $y\in N^+(x_i)$ or $y\in N^-(x_i)$. Therefore, for each color $c_i$, with $i\in[r]$, we only have to choose at most  one new vertex in $B$ and  assign to it color $c_{r+1}$, so that
for each $i\in[r]$ there are vertices of color $c_{r+1}$ in $N^+(x_i)$ and in $N^-(x_i)$. Observe that in this step we colored at most $2r+1$ new vertices in $B$. 
Thus, so far there are at most $(r+1)^2$ colored vertices in $B$. 

 Since $\delta_A\geq n^2$, we can proceed in this way until coloring the vertex $x_n$ with color $c_n$ and coloring some vertices in $B$ in a way that, for each $i\in [n]$, $x_i$ has color $c_i$, and such that there are vertices of colors $\{c_1, \dots, c_n\}\setminus\{c_i\}$ in $N^+(x_i)$ and in $N^-(x_i)$. 
If there are more vertices in $B$, let assign them color $c_1$. This coloring of $D$ is a $b$-coloring with $n$ colors, and the result follows. 

For the general case, colored $x_1$ with color $c_1$ and $x_2$ with color $c_2$. Choose a vertex in $N^+(x_1)$ and other in $N^-(x_1)$  and assign them color $c_2$. Similarly, choose a vertex in $N^+(x_2)$ and other in $N^-(x_2)$  and assign them color $c_1$. So far we have colored at most $4$ vertices in $B$.
Assume that for some $2\leq r<n$, for every $i\in [r]$ $x_i$ receives color $c_i$; for each $i\in[r]$ there are vertices of colors $\{c_1, \dots, c_r\}\setminus\{c_i\}$ in $N^+(x_i)$ and in $N^-(x_i)$; and that there at most $2r(r-1)$ colored vertices in $B$. Since $\delta_A \geq 2n(n-1)$, for each $x_i\in A$, with $i\in [r]$, in $N^+(x_i)$ (as well as in $N^-(x_i)$) there are at least $2n(n-1)-2r(r-1)$ vertices with no color. Let colored $x_{r+1}$ with color $c_{r+1}$, and for each $i\in[r]$ choose a vertex in $N^+(x_{r+1})$ and another in $N^-(x_{r+1})$ and assign them color $c_i$. Also, for each $i\in [r]$, choose a vertex in $N^+(x_{i})$ and another in $N^-(x_{i})$ and assign them color $c_{r+1}$.
Observe that in this step we colored at most $4r$ vertices in $B$, so there are at most $2(r+1)r$ colored vertices in $B$. From here, as in the case of the orientation of $K_{n,m}$, the result follows. \end{proof}

%%%%%%%%%%%%

For the next theorem, given  a bipartite  digraph $D =(A, B)$, and a pair of vertices  $\{y_1, y_2\}\subseteq B$,  let $c(y_1, y_2)= |N_D^+(y_1)\cap N_D^-(y_2)|$.

\begin{theorem}
 Let $D =(A, B)$ be a simple bipartite  digraph, with $|A| = n\leq m = |B|$. If for a pair $\{y_1, y_2\}\subseteq B$, $$2 c(y_1, y_2)^2 < (\frac{m-2}{m-2-(\delta_A-1)})^{\lfloor\frac{m-2}{c(y_1, y_2)}\rfloor} $$
then  $\text{dib}(D)\geq c(y_1, y_2)+1$.
\end{theorem}
\begin{proof} Let $D =(A, B)$ be a simple bipartite  digraph and $\{y_1, y_2\}\subseteq B$ be as in the statement. Let $C= N_D^+(y_1)\cap N_D^-(y_2)\subseteq A$ and let $D'$ the subdigraph of $D$ induced by $C\cup B\setminus \{y_1, y_2\}$. Observe that $|C|=c(y_1, y_2) $, $|B\setminus \{y_1, y_2\}|=m-2$ and $\delta_C\geq \delta_A-1$, thus, by Lemma \ref{partitions},  there is a $b$-coloring $\Gamma': V(D')\rightarrow [c(y_1, y_2)]$ where the  vertices in $C$ receive the colors $\{1, \dots, c(y_1, y_2)\}$ and all of them are $b$-vertices.

Let $\Gamma:V(D)\rightarrow [c(y_1, y_2)+1]$ defined as follows: $\Gamma(x) = \Gamma'(x)$ if $x\in V(D')$; and $\Gamma(x) = c(y_1, y_2)+1$ in other case. Since $C=N_D^+(y_1)\cap N_D^-(y_2)$, all the vertices in $C$ are $b$-vertices in $\Gamma$; and $y_1$ is a $b^+$-vertex and $y_2$ is a $b^-$-vertex. All the chromatic classes of color $i\in [c(y_1, y_2)]$ are stars, thus are acyclic. For the chromatic class of color $c(y_1, y_2)+1$, since the only vertices of color $c(y_1, y_2)+1$ in $B$ are $\{y_1, y_2\}$,  any directed cycle of that color  must be a square and contains the pair of vertices $\{y_1, y_2\}$, which is not possible since none of the vertices in $C= N_D^+(y_1)\cap N_D^-(y_2)$ receives color $c(y_1, y_2)+1$. \end{proof}

%%%%%%%%%%%%%%%%%%%%%%%%%%%%%%%%%%%%%%%%%%%%%%%%%%%%%%%%%%%%%

An orientation of the complete bipartite graph $(A,B)$, with $|A| = n$ and $|B|=m$, such that each arc goes from $A$ to $B$ is denoted by $\di{K_{n,m}}$. 

The following theorem shows that $\text{dib}(D)$ is an increasing function on $min\{n,m\}$, when $D$ is an orientation of the complete bipartite graph $K_{n,m}$.

\begin{theorem}\label{log} Let $D =(A, B)$ be an orientation of  $K_{n,m}$, with $|A| = n\leq m = |B|$, and let $r$ be an integer such that $n\geq r2^r(1+\frac{1}{2^r})^r+2r-1$. Then \[r\leq \text{dib}(D).\]
\end{theorem}
\begin{proof} Let $D =(A , B)$ be an orientation of  $K_{n,m}$, with $|A| = n\leq m = |B|$, and let  $r$ be the greater integer such that there is a subdigraph $D'=(A', B')$ of $D$ isomorphic to 
 $\di{K_{r,r}}$, with $A'= \{x_1, \dots, x_r\}$ and $B'= \{y_1, \dots, y_r\}$.  

Let  $\Gamma : V(D) \rightarrow \{c_1, \dots, c_r\}$ be a coloring  where: for each $i\in [r]$, $\Gamma(x_i) = \Gamma(y_i) = c_i$; for each $x\in A\setminus A'$, $\Gamma(x) = c_1$, and for each $x\in B\setminus B'$, $\Gamma(x) = c_2$. 

Each chromatic class is a star, thus $\Gamma$ is an acyclic coloring, and for each $i\in [r]$, $x_i$ is a $b^+$-vertex and $y_i$ is a $b^-$-vertex. Thus, $\Gamma$ is a $b$-coloring and $r\leq \text{dib}(D)$. 
Therefore, to end the proof we just have to show that if $n\geq r2^r(1+\frac{1}{2^r})^r+2r-1$ then there is a copy of  $\di{K_{r,r}}$ in $D$. Without loss of generality, let us suppose that $\sum\limits_{x\in A}d_D^+(x) \geq \sum\limits_{x\in A}d_D^-(x)$.  

Let $H=(X,Y)$ be the bipartite graph where $X=A$ and $Y$ is the set of $r$-subsets of vertices of $B$. Given $z\in X$ and $S\in Y$, $zS$ is an edge of $H$ if  $S\subseteq N^+_D(z)$. If there is no copy of $\di{K_{r,r}}$, it follows that for every $S\in Y$, $d_H(S)\leq r-1$ which implies that  $\sum\limits_{z\in X} {{d_D^+(z)}\choose {r}} \leq (r-1) {{m}\choose {r}}$. 

 Since $\sum\limits_{x\in A}d_D^+(x) \geq \sum\limits_{x\in A}d_D^-(x)$, it follows that 
 $\sum\limits_{x\in A}d_D^+(x)\geq \frac{nm}{2}$ and therefore  $\sum\limits_{z\in X} {{d_D^+(z)}\choose {r}} \geq n{{\lfloor\frac{m}{2}\rfloor}\choose {r}}$. Hence  $$(r-1){{m}\choose {r}}\geq  n{{\lfloor\frac{m}{2}\rfloor}\choose {r}},$$ which implies that $$\frac{{{m}\choose {r}}}{{{\lfloor\frac{m}{2}\rfloor}\choose {r}}}\geq \frac{n}{r-1}.$$ 
 On the one hand, $\frac{{{m}\choose {r}}}{{{\lfloor\frac{m}{2}\rfloor}\choose {r}}}\leq \big(\frac{m-r+1}{\lfloor\frac{m}{2}\rfloor -r+1}\big)^r= 2^r\big(\frac{m-r+1}{2\lfloor\frac{m}{2}\rfloor -2r+2}\big)^r\leq 2^r\big(\frac{m-r+1}{m -2r+1}\big)^r =  2^r\big(1+\frac{r}{m -2r+1}\big)^r$, and since $m\geq n \geq r2^r(1+\frac{1}{2^r})^r+2r-1$, it follows that $\frac{{{m}\choose {r}}}{{{\lfloor\frac{m}{2}\rfloor}\choose {r}}}\leq 2^r\big(1+ \frac{r}{r2^r(1+\frac{1}{2^r})^r}\big)^r= 2^r\big(1+ \frac{1}{2^r(1+\frac{1}{2^r})^r}\big)^r$. 
 
  On the other hand, since $n\geq r2^r(1+\frac{1}{2^r})^r+2r-1$, it follows that   $\frac{n}{r-1}> 2^r(1+\frac{1}{2^r})^r$.  Thus,  $2^r\big(1+ \frac{1}{2^r(1+\frac{1}{2^r})^r}\big)^r> 2^r(1+\frac{1}{2^r})^r$, which  is not possible, and the result follows. \end{proof}

%%%%%%%%%%%%%%%%%%%%%%%%%%%%%%%%%%%%%%%%%%%%%%%%%%%%%%%%%%%%

\section{Proposed problems}

To conclude, we would like to present a couple of problems that, in our opinion, might be of interest. The first is related to Theorems \ref{teo9}, \ref{teo10}, and \ref{d=3}, and the second is related to tournaments.

\textbf{Problem 1}. Weaken the condition of $\delta_D \geq 2$ when considering bipartite digraphs with $dib(D)\geq 3$, for example, require that $\delta^+_D \geq 1$, $\delta^-_D \geq 1$ and $\delta^+_D + \delta^-_D \geq 3$. In this case, Figure \ref{Fig1} shows an example.

\begin{figure}[htbp!]
\begin{center}
\includegraphics{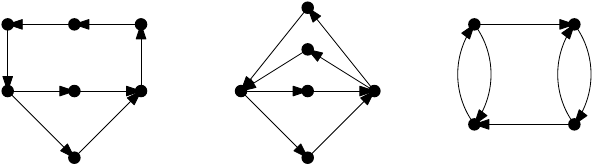}
\caption{\label{Fig1} There are bipartite graphs with 3 components whose $\text{dib}(D)$ is 2.}
\end{center}
\end{figure}

\textbf{Problem 2}. The behavior of the $b$-chromatic number is closely related to the Grundy number, for example \cite{MR4221644,MR4321298}. In the case of the digrundy number $\text{dG}$ \cite{MR4426060} and the dib-chromatic number, we believe that this behavior will hold. 
In particular, given a tournament $T_n$, it was shown in \cite{MR3875016} that the dichromatic number $\text{dac}(T_n)=\Omega(\frac{n}{\lg n})$, such bound is given using an order in the coloring, so the coloring turns out to be greedy, hence $\text{dG}(T_n)=\Omega(\frac{n}{\lg n})$. The question is also whether $\text{dib}(T_n)=\Omega(\frac{n}{\lg n})$.

\section{Statements and Declarations}

This work was partially supported by PAPIIT-M{\' e}xico under Projects IN111626 and IN113324.

%The authors have no relevant financial or non-financial interests to disclose.

All authors contributed to the conception and design of the study. All authors drafted the manuscript, read it, and approved the final manuscript.

%---------------------- Bibliography ---------------------

\bibliographystyle{plain}
\bibliography{biblio}

\begin{thebibliography}{10}

\bibitem{MR3875016}
G.~Araujo-Pardo, J.~J. Montellano-Ballesteros, M.~Olsen, and C.~Rubio-Montiel.
\newblock The diachromatic number of digraphs.
\newblock {\em Electron. J. Combin.}, 25(3):Paper No. 3.51, 17, 2018.

\bibitem{MR4426060}
G.~Araujo-Pardo, J.~J. Montellano-Ballesteros, M.~Olsen, and C.~Rubio-Montiel.
\newblock The digrundy number of digraphs.
\newblock {\em Discrete Appl. Math.}, 317:117--123, 2022.

\bibitem{MR2998438}
K.~J. Edwards.
\newblock Harmonious chromatic number of directed graphs.
\newblock {\em Discrete Appl. Math.}, 161(3):369--376, 2013.

\bibitem{MR593699}
P.~Erd{\H o}s.
\newblock Problems and results in number theory and graph theory.
\newblock In {\em Proceedings of the {N}inth {M}anitoba {C}onference on
  {N}umerical {M}athematics and {C}omputing ({U}niv. {M}anitoba, {W}innipeg,
  {M}an., 1979)}, volume XXVII of {\em Congress. Numer.}, pages 3--21. Utilitas
  Math., Winnipeg, MB, 1980.

\bibitem{MR2895432}
S.~M. Hegde and L.~P. Castelino.
\newblock Further results on harmonious colorings of digraphs.
\newblock {\em AKCE Int. J. Graphs Comb.}, 8(2):151--159, 2011.

\bibitem{MR3329642}
S.~M. Hegde and L.~P. Castelino.
\newblock Harmonious colorings of digraphs.
\newblock {\em Ars Combin.}, 119:339--352, 2015.

\bibitem{MR1670155}
R.~W. Irving and D.~F. Manlove.
\newblock The {$b$}-chromatic number of a graph.
\newblock {\em Discrete Appl. Math.}, 91(1-3):127--141, 1999.

\bibitem{MR841316}
H.~Jacob and H.~Meyniel.
\newblock Extension of {T}ur\'an's and {B}rooks' theorems and new notions of
  stability and coloring in digraphs.
\newblock In {\em Combinatorial mathematics ({M}arseille-{L}uminy, 1981)},
  volume~75 of {\em North-Holland Math. Stud.}, pages 365--370. North-Holland,
  Amsterdam, 1983.

\bibitem{MR2606622}
M.~Jakovac and S.~Klav\v{z}ar.
\newblock The {$b$}-chromatic number of cubic graphs.
\newblock {\em Graphs Combin.}, 26(1):107--118, 2010.

\bibitem{javiernol2024dibchromaticnumberdigraphs}
N.~Javier-Nol, C.~Rubio-Montiel, and I.~Torres-Ramos.
\newblock The dib-chromatic number of digraphs, 2024.

\bibitem{MR3511873}
J.~Kok and N.~K. Sudev.
\newblock The {$b$}-chromatic number of certain graphs and digraphs.
\newblock {\em J. Discrete Math. Sci. Cryptogr.}, 19(2):435--445, 2016.

\bibitem{MR1927071}
M.~Kouider and M.~Mah\'{e}o.
\newblock Some bounds for the {$b$}-chromatic number of a graph.
\newblock {\em Discrete Math.}, 256(1-2):267--277, 2002.

\bibitem{MR2063820}
J.~Kratochv\'{\i}l, Z.~Tuza, and M.~Voigt.
\newblock On the {$b$}-chromatic number of graphs.
\newblock In {\em Graph-theoretic concepts in computer science}, volume 2573 of
  {\em Lecture Notes in Comput. Sci.}, pages 310--320. Springer, Berlin, 2002.

\bibitem{MR4221644}
Zoya Masih and Manouchehr Zaker.
\newblock On {G}rundy and b-chromatic number of some families of graphs: a
  comparative study.
\newblock {\em Graphs Combin.}, 37(2):605--620, 2021.

\bibitem{MR4321298}
Zoya Masih and Manouchehr Zaker.
\newblock Some comparative results concerning the {G}rundy and b-chromatic
  number of graphs.
\newblock {\em Discrete Appl. Math.}, 306:1--6, 2022.

\bibitem{MR693366}
V.~Neumann-Lara.
\newblock The dichromatic number of a digraph.
\newblock {\em J. Combin. Theory Ser. B}, 33(3):265--270, 1982.

\bibitem{MR3202296}
\'E. Sopena.
\newblock Complete oriented colourings and the oriented achromatic number.
\newblock {\em Discrete Appl. Math.}, 173:102--112, 2014.

\end{thebibliography}

\end{document}